\documentclass[12pt,english]{article}
\usepackage{amsfonts, amssymb, amsmath, amsthm, bm, hyperref}
\usepackage{verbatim}

\usepackage{mathtools} % for coloneqq
\newcommand{\coleq}\coloneqq

\hypersetup{
  colorlinks=true,
  allcolors=blue
}

\newcommand{\N}{\mathbb{N}}
\newcommand{\Z}{\mathbb{Z}}
\newcommand{\ud}{\mathrm{d}}
\newcommand{\R}{\mathbb{R}}

\newcommand{\bC}{{\mathbf{C}}}

\newcommand{\bI}{{\mathbf{I}}}
\newcommand{\bJ}{{\mathbf{J}}}

\newcommand{\bP}{{\mathbf{P}}}
\newcommand{\bQ}{{\mathbf{Q}}}
\newcommand{\bR}{{\mathbf{R}}}

\newcommand{\bT}{{\mathbf{T}}}
\newcommand{\bU}{{\mathbf{U}}}
\newcommand{\bV}{{\mathbf{V}}}
\newcommand{\bW}{{\mathbf{W}}}
\newcommand{\bX}{{\mathbf{X}}}
\newcommand{\bY}{{\mathbf{Y}}}
\newcommand{\bZ}{{\mathbf{Z}}}

\newcommand{\cC}{\mathcal{C}}

\newcommand{\cH}{\mathcal{H}}

\newcommand{\cB}{\mathcal{B}}

\newcommand{\norm}[1]{{\lVert#1\rVert}}

\let\Re\relax

\DeclareMathOperator{\Re}{Re}

\DeclareMathOperator{\im}{im}

\newcommand{\bv}{{\mathbf v}}

\newcommand{\bn}{{\mathbf n}}
\newcommand{\bx}{{\mathbf x}}
\newcommand{\dd}[1][x]{\,\operatorname{d}\!#1}

\newcommand{\ddt}{\frac{\dd[]}{\dd[t]}}

\newcommand{\mHC}{{m_{HC}}}

\newcommand{\mC}{{\mathbf{C}}}
\newcommand{\mCH}{\mC_H} % self-adjoint part of operator C
\newcommand{\mCS}{\mC_S} % Skew-adjoint part of operator C

\newcommand{\mI}{{\mathbf{I}}}

\newcommand{\mR}{{\mathbf{R}}}

\newcommand{\sphere}{\mathbb{S}}
\newcommand{\torus}{\mathbb{T}}

\newtheorem{theorem}{Theorem}[section]

\newtheorem{lemma}[theorem]{Lemma}
\newtheorem{corollary}[theorem]{Corollary}

\theoremstyle{definition}
\newtheorem{definition}[theorem]{Definition}

\theoremstyle{remark}
\newtheorem{remark}[theorem]{Remark}

\newtheorem{example} [theorem]{Example}

\makeindex
\begin{document}

\title{Long- and Short-Time Behavior of Hypocoercive Evolution Equations via Modal Decompositions}

\author{F.~Achleitner\thanks{Vienna University of Technology, Institute of Analysis and Scientific Computing, Wiedner Hauptstr. 8-10, A-1040 Wien, Austria, franz.achleitner@tuwien.ac.at},
A.~Arnold\thanks{Vienna University of Technology, Institute of Analysis and Scientific Computing, Wiedner Hauptstr. 8-10, A-1040 Wien, Austria, anton.arnold@tuwien.ac.at} \thanks{corresponding author},
V.~Mehrmann\thanks{Technische Universit\"at Berlin, Institut f.~Mathematik,  MA 4-5, Stra\ss{}e des 17.~Juni 136, D-10623 Berlin, mehrmann@math.tu-berlin.de}, 
and E.A.~Nigsch\thanks{Vienna University of Technology, Institute of Analysis and Scientific Computing, Wiedner Hauptstr. 8-10, A-1040 Wien, Austria, eduard.nigsch@tuwien.ac.at}
}

\date{\today}

\maketitle

\begin{abstract}
The long- and short-time behavior of solutions to dissipative evolution equations is studied by applying the concept of hypocoercivity. Aiming at partial differential equations that allow for a modal decomposition, we compute estimates that are uniform with respect to all modes. While the special example of the kinetic Lorentz equation was treated in %\cite{AcArMeNi25}, 
previous work of the authors,
%the previous work [Achleitner, Arnold, Mehrmann, Nigsch; J.\ Funct.\ Anal., 2025 ], 
that analysis is generalized here to general evolution equations having a scaled family of generators.
\end{abstract}



\section{Introduction}

This paper is concerned with linear, dissipative evolution equations of the form 
\begin{equation} \label{lineveq} 
 \dot x(t) = - \mC x(t), \quad t>0, \quad x(0)=x_0 \in \cH,
\end{equation}
with $x \colon [0,\infty) \to \cH$ and where $\mC$ is a linear operator on a separable Hilbert space $\cH$. 
We will study
%The topic of this paper is 
the long- and short-time decay behavior of solutions $x(t)$, and this will be analyzed using hypocoercivity techniques. In recent years, this has been discussed in detail for many partial differential equations, and in particular for kinetic and Fokker-Planck equations; see \cite{ADSW21,MR3324910,He06,Villani}. The main goal of this paper is to extend the analysis of \eqref{lineveq} to a family of such equations that have the form
\begin{equation} \label{IVP_eta}
\dot x_\eta(t) = - \mC_\eta x(t), \quad x_\eta(0) = x_{\eta, 0} \in \cH,
\end{equation}
with the operators $\bC_\eta :=\bC_H +\eta \bC_S$ and
\[ 
  \bC_H := \frac{1}{2}(\bC + \bC^*) \textrm{ resp.\ } \bC_S := \frac{1}{2} ( \bC - \bC^*) 
\]
denoting the Hermitian and skew-Hermitian parts of $\bC$, respectively. The subsequent analysis will aim for a uniform decay behavior with respect to the scalar parameter $\eta\ge1$. Here, we consider only the case where $\eta$ is taken from some discrete countable set $E$, e.g. $E=\N$, but the extension to continuous values of $\eta$ is straightforward. In many applications, the family of equations \eqref{IVP_eta}, posed on the direct sum of Hilbert spaces, $H=\oplus_{\eta\in E} \cH$, arises from a modal decomposition of an original evolution problem. For typical examples we refer to \cite{AAC16,ADSW21,MR3324910}, where the scaled skew-Hermitian part, $\eta \bC_S$, arises via Fourier transformation of the kinetic transport term. Here $\eta$ is the wave number, and it is discrete for the position variable on a torus (as in \cite{AcArMeNi25}) and continuous for whole space cases (see \cite{BDMMS20}). In the latter case, the restriction $\eta\ge1$ is crucial, since low wave numbers do not give rise to exponential decay.

Let us illustrate this decomposition with a prototypical example taken from \cite[\S6]{AcArMeNi25}:
Consider the decay behavior of the solutions of the Lorentz kinetic equation 
\begin{equation}\label{lorentz}
 \partial_t f +\bv\cdot\nabla_\bx f = \sigma
 \left(\frac{1}{2\pi} \int_{\sphere^{1}} f\,\ud \bv - f \right) 
% \left(\frac{1}{\meas(\sphere^{d-1})} \int_{\sphere^{d-1}} f\,\ud \bv - f \right) 
 =:  \sigma (\tilde f - f),\quad t>0
\end{equation}
for the phase space distribution $f(\bx,\bv,t)$ with $(\bx,\bv,t) \in \torus^2 \times \sphere^{1} \times \R^+$. %, and space dimension $d\ge2$.
Here, $\bv$ is the velocity and $\bx$ is the position variable, $\tilde f$
denotes the mean of $f$ w.r.t.\ the velocity sphere $\sphere^{1}$.
This is a linear Boltzmann equation with collision operator $\cC f \coleq \sigma (\tilde f - f)$, which is local in position $\bx$, and $\sigma>0$ is some relaxation rate. It describes the evolution of free particles (i.e., without external force) moving on the $2$-dimensional
torus $\torus^2$ with speed $1$ (since the particle collisions preserve the kinetic energy and hence $|\bv|$). Its 3D analog was originally considered to model the flow of electrons in a metal \cite{Lo05}. 
It is well-known (see e.g.~\cite[Theorem 3.1]{Golse}, \cite{hankwanleautaud,UPG}) that this equation exhibits exponential convergence to equilibrium.

To obtain qualitative results on the short- and long-time behavior of solutions of \eqref{lorentz}, in \cite{AcArMeNi25} the concept of \emph{hypocoercivity}   \cite{Villani} was employed: A Fourier transformation of  \eqref{lorentz} w.r.t.\ $\bx$ yields the family of mode equations 
\begin{equation}\label{mode_equations}
\partial_t f_\bn + i \bv \cdot \bn f_\bn = \sigma ( \tilde f_\bn - f_\bn),\quad \bn \in \Z^2,\quad t>0.
\end{equation}
While \eqref{lorentz} is posed on $L^2(\torus^2 \times \sphere^{1})$, \eqref{mode_equations} is posed on $H=\bigoplus_{\bn \in \Z^2} \cH$, with $\cH=L^2(\sphere^{1})$, and both spaces are isomorphic. 

Then, hypocoercivity methods (extensions of \cite{AAC16,AAM23ZAMM,Villani}) were applied to \eqref{mode_equations} to obtain short- and long-time decay estimates of $f_\bn(t)$, uniformly in $\bn$: 
\begin{lemma}\label{lem1}
For each mode $\bn = (n_1, n_2) \in \Z^2 \setminus\{0\}$ the following holds:
\begin{enumerate}
    \item[(a)] The hypocoercivity index (defined in \eqref{Op:hypocoercive:kappa} below) is 1 for each mode.
    \item[(b)] The norm of the solution to~\eqref{mode_equations} decays exponentially for long times like
    \begin{equation*}
        \|f_\bn(t)\|_{L^2(\sphere^{1})}
        \le \min\Big[1,\sqrt{\tfrac{2|\bn|+1}{2|\bn|-1}} e^{-\lambda_0 t} \Big] \|f_\bn(0)\|_{L^2(\sphere^{1})} \,,   
%        \le \sqrt3 e^{-\lambda_0 t}\|f_\bn(0)\|_{L^2(\sphere^{1})} %\,,\quad t\ge0\,,
    \end{equation*}
    for $t\ge0$ with some explicit decay rate $\lambda_0>0$.
%   the (non-sharp) rate $\lambda_0=\frac12-\frac{1}{6\sqrt2}-\frac13\sqrt{\frac{7}{16}+\frac{1}{\sqrt8}}\approx0.0857$.
    \item[(c)] The propagator norm pertaining to \eqref{mode_equations} decays (algebraically) for short times like
\begin{equation*}
    \|\bP_\bn(t)\|_{\cB(L^2(\sphere^1))} \le 1-c t^3,\quad 0\le t\le \tau,
\end{equation*}
with some explicit, $\bn$-independent constants $c,\,\tau>0$. 
\end{enumerate}
\end{lemma}
These combined estimates then led to analogous decay estimates for the original Lorentz equation \eqref{lorentz}. We note that this \emph{mode-by-mode hypocoercivity method} was already used, e.g.\ in \cite[\S II]{ADSW21}, for the long-time analysis of dissipative kinetic equations. But the corresponding short-time analysis was first studied
%is apparently original to 
in \cite{AcArMeNi25}. 

In this paper we will show that these methods are not restricted to the example of the Lorentz equation but can be generalized to other classes of equations that allow for a similar modal decomposition --- in the form of \eqref{IVP_eta}. The proof of the long-time behavior (in \S\ref{sec:long}) and the short-time behavior (in \S\ref{sec:short}) is based on first representing the operator $\bC_\eta$ in \emph{staircase form}, see \cite{AAM23ZAMM}, e.g. Similar  block decompositions of operators into their bounded and unbounded parts also turned out to be advantageous in \cite{LHC2020}.

%Hence, our aim is to develop a framework which can treat a) higher kernel dimension, b) higher hypocoercivity index (although we treat only index 1 here) and c) a large (?) class of equations.

%Behind such a spectral decomposition is an orthogonal decomposition of the Hilbert space $\cH$. For example, on a compact manifold $M$ with boundary, $L^2(\cH)$ can be decomposed into the eigenspaces of a positive and elliptic pseudo-differential operator on $M$.

%Typical examples are the decomposition of $L^2(\R^n)$ into eigenspaces of the harmonic/anharmonic oscillators, the fractional relativistic Schrödinger operators, the relativistic Schrödinger operators; $L^2(\Omega)$ for $\Omega$ bounded in the eigenspaces of the spectral fractional Laplacian, etc.; \cite{CarDGR23}.

The paper is organized as follows. After introducing the preliminaries in Section~\ref{sec:prelim}, we present the main results in Section~\ref{sec:main}. 
We close with a summary and some open questions.

% indexabschätzung gibt automatisch lyapunv-funktional

%%%%%%%%%%%%%%%%%%%%%%%%%%%%%%%%%%%%%%%%%%%%%%%%%%%%%

\section{Notation and Preliminaries}\label{sec:prelim}
We consider operators of the form $\bC = \bR - \bJ$, where the operators $\bR := \bC_H$ and $-\bJ := \bC_S$
have the same domains and form the self-adjoint (Hermitian) and skew-adjoint (skew-Hermitian) part of $\bC$, respectively. If $\bC$ is bounded, then the domains of $\bC_H$ and $\bC_S$ are trivially identical and equal to $\cH$, but 
assuming equality of domains will eventually allow us to generalize our setup to the case of unbounded $\bC$. 
Typically, we use $\bR$ and $\bJ$ instead of $\bC_H$ and $-\bC_S$ to improve the readability of complicated expressions. This is the common notation used for \emph{dissipative operators}, see, e.g. \cite{JacZ12,SchJ14}, where $\bC$ is \emph{accretive}, i.e. $\bR=\bC_H$ is positive semi-definite.

We will analyze the decay behavior of solutions using the concept of \emph{hypocoercivity} which was 
introduced in \cite{Villani} for the study of evolution equations of this form for which the (possibly unbounded) dissipative operator~$-\mC$ generates a \emph{uniformly exponentially stable $C_0$-semigroup} $(e^{-\mC t })_{t\geq 0}$; see e.g.~\cite[Section V.1, Eq.~(1.9)]{EngelNagel2000}.

\begin{definition}\label{def:hypoco}
Let $\bC$ be an (unbounded) operator on a separable Hilbert space $\cH$ generating a strongly continuous semigroup $(e^{-t\bC})_{t \ge 0}$, and let $\widetilde \cH$ be a Hilbert space continuously and densely embedded in $(\ker \bC)^\perp$, endowed with a Hilbertian norm $\norm{\cdot}_{\widetilde \cH}$. The operator $\bC$ is said to be \emph{hypocoercive} on $\widetilde \cH$ if there exists a finite constant $C$ and some $\lambda>0$ such that
\begin{equation}\label{HC-decay}
  \mbox{ for all } x_0 \in \widetilde \cH,\quad \mbox{ for all } t \ge 0:\quad \norm{e^{-t\bC} x_0}_{\widetilde \cH} \le Ce^{-\lambda t}\norm{x_0}_{\widetilde \cH}. 
\end{equation}
\end{definition}

In what follows, we will assume that we are working directly on $(\ker \bC)^\perp$ and hence $\ker \bC = \{0\}$, so we write $\cH$ in place of $\tilde \cH$.

Let us fix some notation for the remainder of this article. $\cH$ will denote a separable Hilbert space and $\cB(\cH)$ the space of bounded linear operators on $\cH$. An operator $\bC \in \cB(\cH)$ is called \emph{accretive} if $\Re  \langle \bC x, x \rangle \ge 0$ for all $x \in \cH$, i.e., the symmetric part of $\bC$ is positive semi-definite.
The~\emph{hypocoercivity index} (HC-index)~$m_{HC}=m_{HC}(\mC)$ of an accretive operator $\mC\in\cB(\cH)$ is defined as the smallest integer~$m\in\N_0$ (if it exists) such that 
\begin{equation}\label{Op:hypocoercive:kappa}
  \sum_{j=0}^m (\mC^*)^j \mCH \mC^j \ge\kappa\mI 
\end{equation}
for some $\kappa>0$.

\begin{remark}\label{equiv-HCcond}
Using the equivalence of the conditions in \cite[Lemma 2.6]{AcArMeNi25},  this definition could also be based on the coercivity of $\sum_{j=0}^m \mCS^j \mCH (\mCS^*)^j$.
\end{remark}

As a practical and just newly found consequence, condition \eqref{Op:hypocoercive:kappa} directly yields a strictly decaying Lyapunov functional for \eqref{IVP_eta}, namely
\begin{equation}\label{P-lyap}
  \|x\|_\bP^2 := \langle x,\bP x\rangle_\cH,
\end{equation}

with the bounded operator $\bP:=\sum_{j=0}^m (\mC^*)^j  \mC^j \ge\mI$. This is easily verified by
\begin{equation}\label{P-est}
  \ddt \|x\|_\bP^2 =-2 \sum_{j=0}^m \langle\mC^j x, \mCH \mC^j x\rangle_\cH \le -2\kappa \|x\|_\cH^2 \le -\frac{2\kappa}{\|\bP\|_{\cB(\cH)}} \|x\|_\bP^2 \,.
\end{equation}
The following example illustrates that this Lyapunov functional \eqref{P-lyap} typically does \emph{not} yield the sharp decay rate, even for simple matrix cases.
\begin{example}
The matrix $\bC:=\begin{bmatrix}
1 & -1 \\
1 & 0
\end{bmatrix}$ 
has the eigenvalues $\lambda =\frac12 \pm i\frac{\sqrt3}{2}$, $m_{HC}=1$,  and $\kappa=1$. It implies $\bP=\begin{bmatrix}
3 & -1 \\
-1 & 2
\end{bmatrix}$ with $\|\bP\|_2=\frac{5+\sqrt5}{2}$. Hence, \eqref{P-est} yields decay of $\|x(t)\|_\bP$ with the sub-optimal rate $\frac{2}{5+\sqrt5}\approx 0.276 < \frac12=\Re(\lambda)$.  \hfill\qed
\end{example}

For any choice of the parameter $\eta$, we set
\begin{equation} 
\bC_\eta = \bC_H + \eta \bC_S = \bR - \eta \bJ. \label{split_eta}
\end{equation}
We call $\bP(t):=e^{-t\bC}$ the \emph{propagator} of $-\bC$ and denote by $\bP_\eta$ the propagator of $-\bC_\eta$.

%%%%%%%%%%%%%%%%%%%%%%%%%%%%%%%%%

\section{Main Result}\label{sec:main}

As a starting point for our analysis, consider an accretive operator $\bC\in\cB(\cH)$ with hypocoercivity index $\mHC=1$. Here, $\bC$ can be understood as pertaining to the mode $\eta=1$ from \eqref{IVP_eta}. Hence, $\bC$ is hypocoercive (see \cite[Theorem 4.6]{AcArMeNi25}) and its propagator $\bP(t)$ in norm decays exponentially for long times, see \eqref{HC-decay}. Furthermore, it exhibits algebraic short-time decay like 
$$
  \|\bP(t)\| = 1-ct^3+o(t^3) \quad \mbox{for } t\to0,
$$
see Theorem 4.1 in \cite{AcArMeNi25}.

The goal of this paper is to prove that the family $\bP_\eta(t)$, $t\ge0$ obeys analogous long- and short-time decay estimates, uniformly in $\eta\ge1$. For simplicity, we discuss here only the case $\mHC=1$ (which occurs for the Lorentz equation, see \cite[\S6]{AcArMeNi25}), but we expect the same behavior to also hold for larger hypocoercivity indices. 

In the following, we make two assumptions for the self-adjoint operator $\mR$. Without loss of generality we may assume that $\|\mR\|=1$. This can always be achieved by an appropriate scaling of the time. We also assume that the uniform bound $\mR \ge \gamma \mI$, with some $\gamma>0$, holds on $(\ker \mR)^\perp$.

Our main result is the following theorem, which generalizes Lemma 6.2 from \cite{AcArMeNi25} to more generic evolution equations.

\begin{theorem}\label{thm1}
Let $\bC\in\cB(\cH)$ be an accretive operator with $\|\bC_H\|=1$, hypocoercivity index $\mHC=1$, i.e., there exists $\kappa>0$ such that $\bC_H +\bC^* \bC_H \bC \geq \kappa \bI$, and let the kernel of the Hermitian part $\bC_H$ be finite dimensional, i.e., $\dim \ker \bC_H <\infty$ and satisfy $\bC_H \ge \gamma \mI$ on $(\ker \bC_H)^\perp$ for some $\gamma>0$.
Then, the family of operators $\bC_\eta$, $\eta\geq 1$ as in \eqref{split_eta} satisfies the following assertions:
\begin{enumerate}
    \item[(a)] The hypocoercivity index satisfies $m_{HC}(\bC_H + \eta\bC_S) = 1$ uniformly in $\eta\ge1$. 
    \item[(b)] The norm of the solution to~\eqref{IVP_eta} %~\eqref{lineveq} with $\bC =\bC_\eta$ 
    decays exponentially for long time like
    \begin{equation}\label{unif-decay}
        \|x_\eta(t)\|_{\cH}
        \le \min\Big[1,\sqrt{\tfrac{\eta+\alpha}{\eta-\alpha}} e^{-\lambda_\eta t} \Big] \|x_\eta(0)\|_{\cH}, \quad t\ge0,    
    \end{equation}
    where the (non-sharp) rate $\lambda_{\eta}\geq \lambda_0>0$ and $\alpha\in(0,1)$ are specified in the proof, see \eqref{lambda0+alpha}.
    \item[(c)] The propagator norms decay (algebraically) for short time like
\begin{equation}\label{Pn-decay}
    \|\bP_\eta(t)\|_{\cB(\cH)} \le 1-c t^3,\quad 0\le t\le \tau,
\end{equation}
and the $\eta$-independent constants $c,\,\tau>0$ are given explicitly in the proof, 
%see \eqref{def-c} and \eqref{def-tau}, respectively.   
see \eqref{def-c1} and \eqref{def-tau1}, respectively.
\end{enumerate} 
\end{theorem}
\begin{remark}
An accretive operator $\bC\in\cB(\cH)$ with hypocoercivity index~$\mHC=1$ may not satisfy the condition $\bC_H \ge \gamma \mI$ on $(\ker \bC_H)^\perp$ for some $\gamma>0$, see the example given in~\cite[Remark 2.5(i)]{AcArMeNi25}.    
\end{remark}
\begin{proof}[Proof of Theorem~\ref{thm1}~(a)]
The assumption that $\mHC(\bC)=1$ and the Remark \ref{equiv-HCcond} imply that $\bR + \bJ \bR \bJ^* \ge \kappa_1 \bI$ for some $\kappa_1>0$. Hence 
\begin{equation} \label{est1b}
 \bR + \eta^2 \bJ \bR \bJ^* 
\geq \bR + \bJ \bR \bJ^* 
\geq \kappa_1 \bI
\end{equation}
proves statement $(a)$ for $\eta\geq 1$.

The proofs of parts $(b)$ and $(c)$ will be presented in the following two subsections.
\end{proof}

The decay behavior of all modes directly translates into a collective decay of the whole system described by $\bx(t):=\big(x_\eta(t)\big)_{\eta\in E}$ in $H=\bigoplus_{\eta \in E} \cH$ using $\|\bx\|_H^2 = \sum_{\eta\in E} \|x_\eta\|_\cH^2$:
%\todo{EN: bzw. $\int |x_\eta|^2$ für stetigen parameter}
%
\begin{corollary}
    Under the assumptions of Theorem \ref{thm1}, the solution of \eqref{IVP_eta} satisfies
\begin{eqnarray*}
    \|\bx(t)\|_{H}
    &\le& \min\Big[1,\sqrt{\tfrac{1+\alpha}{1-\alpha}} e^{-\lambda_0 t} \Big] \|\bx(0)\|_{H}, \quad t\ge0,  \\
    \|\bx(t)\|_{H} 
    &\le& (1-c t^3) \, \|\bx(0)\|_{H},\quad 0\le t\le \tau,
\end{eqnarray*}  
with $\lambda_0$ given in~\eqref{lambda0} below.
%\todo{ev. + lower bound}
\end{corollary}

%%%%%%%%%%%%%%%%%%%%%%%%%%%%%%%%%%%%%%%%%%%%%%%%%%%%%%

\subsection{Proof of long-time behavior}\label{sec:long}
\begin{proof}[Proof of Theorem~\ref{thm1} (b)]
First, we derive a suitable representation of the accretive bounded operator $\bC$ in \emph{staircase form} (see Equation~\eqref{SF_refined} below).
Then, we  construct a positive self-adjoint operator $\bY\in\cB(\cH)$ such that $\|h\|_\bY^2 :=\langle h,\bY h\rangle$ is a strict Lyapunov functional for the evolution of~\eqref{lineveq}.\\

\noindent
\underline{Step 1 (Derivation of the staircase form):} We first recall the staircase form of $\bC=\bR-\bJ$, which corresponds to the mode $\eta=1$ in \eqref{IVP_eta}. This will then carry over verbatim to $\bC_\eta=\bR-\eta\bJ$. 
Due to~\cite[Lemma 5.1]{AcArMeNi25}, an accretive operator $\bC=\bR-\bJ\in \cB(\cH)$ with $\mHC(\bC)=1$ has the following representation:

For $\bT\in\cB(\cH)$, we recall the identities $\ker \bT = (\im \bT^*)^\perp$, $(\ker \bT)^\perp = \overline{\im \bT^*}$. 
Then, we view $\bR$ as an operator $\bR \colon (\ker \bR)^\perp \oplus \ker \bR \to \overline{\im \bR} \oplus (\im \bR)^\perp$
and set
\[
\cH^1_1 \coleq (\ker \bR)^\perp = \overline{\im \bR}, 
\qquad \quad
\cH^1_2 \coleq \ker \bR = (\im \bR)^\perp
\]
to write $\bR, \bJ \in \cB(\cH^1_1 \oplus \cH^1_2)$ in components as follows:
\begin{equation} \label{SF_original} 
\bR = \begin{bmatrix}
\bR_{1,1}^1 & 0 \\
0 & 0
\end{bmatrix}, \qquad
\bJ = \begin{bmatrix}
\bJ^1_{1,1} & \bJ^1_{1,2} \\
\bJ^1_{2,1} & \bJ^1_{2,2}
\end{bmatrix}.
\end{equation}
Then by the assumption on $\bR=\bC_H$ that $\bR\ge\gamma \bI$ on $(\ker \bR)^\perp$,  we have $\bR_{1,1}^1\geq \gamma \bI$ on $\cH_1^1$ for some $\gamma>0$.%\todo{Voraussetzung aufnehmen, damit das (was genau?) geht; Ist der Normalfall, pathologisches Gegenbeispiel angeben. VM: Ich denke dass sollten wir vielleicht lassen.}

As in~\cite[Lemma 1]{AAM23ZAMM}, we decompose further: $\cH_1^1 =\cH_0 \oplus \cH_1$ where 
\[
 \cH_0 \coleq \ker \bJ_{2,1}^1,
 \qquad \quad
 \cH_1 \coleq \cH_0^\perp \quad \text{(in $\cH_1^1$)},
 \qquad \quad
 \cH_2 :=\cH_2^1,
\]
such that $\bJ_{2,1}^1$ has the representation 
\[
 \bJ_{2,1}^1 \colon 
  \cH_1^1 =\cH_0 \oplus \cH_1 \to \cH_2^1, \quad
  \bJ_{2,1}^1 =[0\quad \bJ_{2,1}].
\]
Here, we have  $\bJ_{2,1}\colon\cH_1 \to\cH_2$ and
\begin{equation}\label{dimfinite}
\dim \cH_1 =\dim \cH_2 < \infty.
\end{equation}
Hence, $\bJ_{2,1}$ can be represented by a square matrix.
Due to the hypocoercivity of~$\bC$ and~\cite[Remark 5.2]{AcArMeNi25}, the matrix $\bJ_{2,1}$ is nonsingular.

Using the decomposition $\cH_1^1 =\cH_0 \oplus \cH_1$ such that $\cH =\cH_1^1 \oplus \cH_2^1 = \cH_0 \oplus \cH_1 \oplus \cH_2$, we refine the staircase form~\eqref{SF_original} and obtain
\begin{equation} \label{SF_refined} 
\bR = \left[\begin{array}{cc|c}
\bR_{0,0} & 0 & 0 \\
0 & \bR_{1,1} & 0 \\
\hline
0 & 0 & 0
\end{array}\right], \qquad
\bJ = \left[\begin{array}{cc|c}
\bJ_{0,0} & -\bJ_{1,0}^* & 0 \\
\bJ_{1,0} & \bJ_{1,1} & -\bJ_{2,1}^* \\
\hline
0 & \bJ_{2,1} & \bJ_{2,2}
\end{array}\right]\,,
\end{equation}
where $\bJ_{2,2} =\bJ^1_{2,2}$.
For future reference, we recall that $\cH_0$ may be infinite-dimensional.\\

\noindent
\underline{Step 2 (Ansatz  and spectrum of $\bY_\eta$):}
Due to Theorem \ref{thm1}(a), the operators $\bC_\eta =\bR-\eta\bJ$, $\eta\geq 1$ are hypocoercive with HC-index $\mHC(\bC_\eta)=1$. 
Consider for some $\epsilon>0$ the ansatz
\begin{equation} \label{Y_n}
 \bY_\eta
:=\left[\begin{array}{cc|c}
 \bI & 0 & 0 \\
 0 & \bI & \tfrac{\epsilon}{\eta} \bJ_{2,1}^* \\
\hline 
 0 & \tfrac{\epsilon}{\eta} \bJ_{2,1} & \bI  
 \end{array}\right]\,,
\end{equation}
where %(the short-hand notation) 
$\bI$ denotes the identity on the respective Hilbert spaces $\cH_i$, $i=0,1,2$.
For all $\eta\geq 1$, the bounded operators $\bY_\eta$ are self-adjoint.

Moreover, if $\epsilon>0$ is sufficiently small then the operators $\bY_\eta$, $\eta\geq 1$ are positive. 
To study the spectrum of $\bY_\eta$, we consider the representation 
\begin{equation} \label{Y_n+X_n}
 \bY_\eta
=\left[\begin{array}{cc}
 \bI_{\cH_0} & 0 \\
 0 & \bX_\eta 
 \end{array}\right]
\qquad\text{with } 
 \bX_\eta
=\left[\begin{array}{cc}
 \bI & \tfrac{\epsilon}{\eta} \bJ_{2,1}^* \\
 \tfrac{\epsilon}{\eta} \bJ_{2,1} & \bI  
 \end{array}\right]
 =\bI +\tfrac{\epsilon}{\eta} 
\underbrace{\begin{bmatrix}
 0 & \bJ_{2,1}^* \\
 \bJ_{2,1} & 0  
 \end{bmatrix}}_{=:\bZ_\eta} \,,
\end{equation}
where each block matrix of $\bX_\eta\in\cB(\cH_1\oplus\cH_2)$ is of dimension $n:=\dim \cH_1 =\dim \cH_2 \in\N$. 
Hence, each $\bY_\eta$, $\eta\ge1$ only has a pure point spectrum: the eigenvalue 1 due to $\bI_{\cH_0}$ and the eigenvalues of $\bX_\eta$.

It is straightforward to see that, for each eigenvalue $\lambda$ of the Hermitian matrix $\bZ_\eta$, $\lambda^2$ is an eigenvalue of $\bJ_{2,1}\bJ_{2,1}^*$. %\makered{using that $\sigma\big(\bJ_{2,1}\bJ_{2,1}^*\big) =\sigma\big(\bJ_{2,1}^*\bJ_{2,1}\big)$, see~\cite[Theorem 1.3.22]{HoJo13}}. 
This yields the following estimate for the eigenvalues $\lambda_i(\bZ_\eta)$ of $\bZ_\eta$:
\[ 
 -\frac{\epsilon}{\eta} \sqrt{\lambda_{\max}\big(\bJ_{2,1}\bJ_{2,1}^*\big)}
\leq \frac{\epsilon}{\eta} \lambda_i(\bZ_\eta)
\leq \frac{\epsilon}{\eta} \sqrt{\lambda_{\max}\big(\bJ_{2,1}\bJ_{2,1}^*\big)}\,,
\]
and hence
\[ 
 1-\frac{\epsilon}{\eta} \sqrt{\lambda_{\max}\big(\bJ_{2,1}\bJ_{2,1}^*\big)}
\leq \lambda_i(\bX_\eta)
\leq 1 +\frac{\epsilon}{\eta} \sqrt{\lambda_{\max}\big(\bJ_{2,1}\bJ_{2,1}^*\big)}
\]
for $i=1,\ldots,2n$.
Thus we obtain the estimates
\begin{equation} \label{EV_Yn}
 0<1-\frac{\epsilon}{\eta} \sqrt{\lambda_{\max}\big(\bJ_{2,1}\bJ_{2,1}^*\big)}
\leq \lambda_{\min}(\bY_\eta)
\leq \lambda_{\max}(\bY_\eta)
\leq  1+\frac{\epsilon}{\eta} \sqrt{\lambda_{\max}\big(\bJ_{2,1}\bJ_{2,1}^*\big)}
\end{equation}
for all $\eta\geq 1$. For the proof of positivity of the self-adjoint operators $\bY_\eta$, $\eta\geq 1$, we have used here the sufficient condition 
\begin{equation} \label{condition_0}
 \frac14 -\epsilon^2 \lambda_{\max}\big(\bJ_{2,1} \bJ_{2,1}^*\big) \ge 0 \,.
% 1 -\epsilon^2 \lambda_{\max}\big(\bJ_{2,1} \bJ_{2,1}^*\big) > 0 \,.
\end{equation}
For later usage in \eqref{lambda0+alpha}, this condition is stricter than needed in \eqref{EV_Yn}.

\smallskip
\noindent
\underline{Step 3 (Checking the Lyapunov matrix inequality):} 
We show that, for sufficiently small $\epsilon>0$, there exists a constant $\kappa_2>0$ (independent of $\eta$) such that
\begin{equation} \label{LMI}
 \bQ_\eta :=\bC_\eta^* \bY_\eta +\bY_\eta \bC_\eta \geq 2\kappa_2 \bI 
\qquad \text{for all $\eta\geq 1$.}
\end{equation}
To derive sufficient conditions on $\epsilon>0$ and $\kappa_2>0$, we check the uniform positivity of the self-adjoint operator $\bQ_\eta -2\kappa_2 \bI$, $\eta\geq 1$ using the characterization via Schur complements, see e.g.~\cite{St05,Tr08}.
The self-adjoint operator $\bQ_\eta -2\kappa_2 \bI$, $\eta\geq 1$ is given as 
\begin{equation*}
\begin{split}
& \bQ_\eta -2\kappa_2 \bI
\\
&=\bC_\eta^* \bY_\eta +\bY_\eta \bC_\eta -2\kappa_2 \bI
\\
&=\left[\begin{array}{cc|c}
 2\bR_{0,0} -2\kappa_2\bI & 0 & -\epsilon \bJ_{1,0}^* \bJ_{2,1}^* 
\\
 0 & 2\bR_{1,1} -2\epsilon\bJ_{2,1}^* \bJ_{2,1} -2\kappa_2\bI & \epsilon\big(\bJ_{1,1} +\tfrac1{\eta} \bR_{1,1}\big) \bJ_{2,1}^* -\epsilon \bJ_{2,1}^* \bJ_{2,2} 
\\
\hline
 -\epsilon\bJ_{2,1}\bJ_{1,0} & \epsilon\bJ_{2,1} \big(-\bJ_{1,1} +\tfrac1{\eta} \bR_{1,1}\big) +\epsilon\bJ_{2,2}\bJ_{2,1} & 2\epsilon \bJ_{2,1}\bJ_{2,1}^* -2\kappa_2\bI 
 \end{array}\right]
\\
&=:\begin{bmatrix} \bV & \bW^* \\ \bW & \bU \end{bmatrix} .
\end{split}
\end{equation*}
This block operator is positive if and only if $\bU$ and the Schur complement $(\bQ_\eta -2\kappa_2\bI)/\bU =\bV -\bW^*\bU^{-1}\bW$ are positive (for the finite dimensional analog see~\cite[Theorem 1.12]{Zh05}, \cite[Prop. 10.2.5]{Be18}).
To this end we derive two conditions on $\epsilon>0$ and $\kappa_2>0$: %to ensure the positivity of the self-adjoint operator $\bQ_\eta -2\kappa_2 \bI$: 

On the one hand, we consider the operator $\bU =2\epsilon \bJ_{2,1}\bJ_{2,1}^* -2\kappa_2\bI \in\cB(\cH_2)$ on the finite-dimensional Hilbert space~$\cH_2$.
Since $\bJ_{2,1}:\cH_1\to\cH_2$ is nonsingular, the self-adjoint operator $\bJ_{2,1}\bJ_{2,1}^* \in\cB(\cH_2)$ is positive and satisfies $\bJ_{2,1}\bJ_{2,1}^* \geq \lambda_{\min}(\bJ_{2,1}\bJ_{2,1}^*) \bI$, where $\lambda_{\min}(\bJ_{2,1}\bJ_{2,1}^*)>0$ is the smallest eigenvalue of $\bJ_{2,1}\bJ_{2,1}^*$. Consequently,
\begin{equation} \label{LB_U}
 \bU 
=2\epsilon \bJ_{2,1}\bJ_{2,1}^* -2\kappa_2\bI
\geq 2\big(\epsilon \lambda_{\min}(\bJ_{2,1}\bJ_{2,1}^*) -\kappa_2\big) \bI .
\end{equation}
Thus, we obtain the condition
\begin{equation} \label{condition_1}
    \epsilon \lambda_{\min}(\bJ_{2,1}\bJ_{2,1}^*) -\kappa_2 > 0.
\end{equation}
To fulfill condition~\eqref{condition_1} set  
\begin{equation} \label{kappa_delta}
% \kappa_\delta
 \kappa_2
:= \delta \epsilon \lambda_{\min}(\bJ_{2,1}\bJ_{2,1}^*) 
\end{equation}
for some $\delta\in(0,1)$ that is fixed from now on.

Next, we check the positivity of the complement $(\bQ_\eta -2\kappa_2\bI)/\bU =\bV -\bW^*\bU^{-1}\bW$.
First, we consider $-\bW^*\bU^{-1}\bW$ and, using~\eqref{LB_U}--\eqref{condition_1}, we estimate:
\[
 -\bW^*\bU^{-1}\bW
\geq -\frac{1}{2\big(\epsilon \lambda_{\min}(\bJ_{2,1}\bJ_{2,1}^*) -\kappa_2\big)} \bW^* \bW
\geq -\frac{\epsilon^2 \omega}{2\big(\epsilon \lambda_{\min}(\bJ_{2,1}\bJ_{2,1}^*) -\kappa_2\big)} \bI,
\]
where $\omega>0$ is chosen such that 
\begin{align*}
\bW^* \bW % & \\
&=\epsilon^2 \begin{bmatrix} -\bJ_{1,0}^* \bJ_{2,1}^* 
\\ \big(\bJ_{1,1} +\tfrac1{\eta} \bR_{1,1}\big) \bJ_{2,1}^* - \bJ_{2,1}^* \bJ_{2,2} \end{bmatrix}
 \begin{bmatrix}
    -\bJ_{2,1}\bJ_{1,0} & \bJ_{2,1} \big(-\bJ_{1,1} +\tfrac1{\eta} \bR_{1,1}\big) +\bJ_{2,2}\bJ_{2,1} 
 \end{bmatrix}
\\
&\leq \epsilon^2 \omega \bI
\end{align*}
for all $\eta\geq 1$.
Finally, we consider the Schur complement
\begin{align*}
 (\bQ_\eta -2\kappa_2\bI)/\bU %\\
&=\bV -\bW^*\bU^{-1}\bW
\\
&=\begin{bmatrix}
 2\bR_{0,0} -2\kappa_2\bI & 0 
\\
 0 & 2\bR_{1,1} -2\epsilon\bJ_{2,1}^* \bJ_{2,1} -2\kappa_2\bI 
 \end{bmatrix}
 -\bW^*\bU^{-1}\bW
\\
&\geq \begin{bmatrix}
 2(\gamma -\kappa_2)\bI & 0 
\\
 0 & 2(\gamma -\epsilon\lambda_{\max}(\bJ_{2,1}^* \bJ_{2,1}) -\kappa_2)\bI 
 \end{bmatrix}
 -\frac{\epsilon^2 \omega}{2\big(\epsilon \lambda_{\min}(\bJ_{2,1}\bJ_{2,1}^*) -\kappa_2\big)} \bI ,
\end{align*}
where in the last estimate we have used that $\bR_{1,1}^1>\gamma\bI$ for some $\gamma>0$.
Therefore, we have the sufficient conditions 
\begin{subequations}
\begin{align}  \label{condition_2a}
0 &< 2(\gamma -\kappa_2) -\frac{\epsilon^2 \omega}{2\big(\epsilon \lambda_{\min}(\bJ_{2,1}\bJ_{2,1}^*) -\kappa_2\big)}
\intertext{and} \label{condition_2b} 
0 &< 2(\gamma -\epsilon\lambda_{\max}(\bJ_{2,1}^* \bJ_{2,1}) -\kappa_2) -\frac{\epsilon^2 \omega}{2\big(\epsilon \lambda_{\min}(\bJ_{2,1}\bJ_{2,1}^*) -\kappa_2\big)} .
\end{align}
\end{subequations}
But condition~\eqref{condition_2b} already implies condition~\eqref{condition_2a}, since $\epsilon\lambda_{\max}(\bJ_{2,1}^* \bJ_{2,1})>0$.
%Setting $\kappa_2=\kappa_\delta$ in
With \eqref{kappa_delta}, 
%for some fixed $\delta\in(0,1)$ \todo{EN: das wurde doch eh schon weiter oben fixiert?}, 
the sufficient condition~\eqref{condition_2b} simplifies to 
\begin{equation} \label{condition_2}
\begin{split}
 0 
&< 2(\gamma -\epsilon\lambda_{\max}(\bJ_{2,1}^* \bJ_{2,1}) -\kappa_2) -\frac{\epsilon^2 \omega}{2\big(\epsilon \lambda_{\min}(\bJ_{2,1}\bJ_{2,1}^*) -\kappa_2\big)}
\\
&= 2\gamma -\epsilon\Big(2\lambda_{\max}(\bJ_{2,1}^* \bJ_{2,1}) +2\delta \lambda_{\min}(\bJ_{2,1}\bJ_{2,1}^*) +\frac{\omega}{2 (1-\delta) \lambda_{\min}(\bJ_{2,1}\bJ_{2,1}^*)}\Big) .
\end{split}
\end{equation}
So, altogether 
we are left with the two conditions~\eqref{condition_0} and \eqref{condition_2} for $\epsilon$.
Since $\gamma>0$, those two conditions are satisfied for any $\epsilon>0$ sufficiently small. Hence, we have proved the positivity of $\bQ_\eta -2\kappa_2 \bI$ and therefore the Lyapunov inequality \eqref{LMI}. \\

\noindent
\underline{Step 4 (Proof of the long-time behavior):} 
Consider a solution $x_\eta(t)$ of the initial value problem~\eqref{IVP_eta} %~\eqref{lineveq} with $\bC=\bC_\eta$ 
for any $\eta\geq 1$.
Then,  using $\bY_\eta$ from~\eqref{Y_n}, the derivative of the weighted norm $\|x_\eta(t)\|_{\bY_\eta}^2$ satisfies
\begin{equation} \label{DI}
 \begin{split}
 \ddt \|x_\eta(t)\|_{\bY_\eta}^2     
&=-\langle x_\eta(t), \big(\bC^* \bY_\eta +\bY_\eta \bC\big) x_\eta(t)\rangle
\\
&\leq -2\kappa_2 \langle x_\eta(t),x_\eta(t) \rangle
% \leq -2\frac{\kappa_2}{\lambda_{\max}(\bY_\eta)} \langle x_\eta(t),\bY_\eta x_\eta(t) \rangle
 \leq -2\frac{\kappa_2}{\lambda_{\max}(\bY_\eta)} \|x_\eta(t)\|_{\bY_\eta}^2,
 \end{split}    
\end{equation}
where we have employed  the Lyapunov matrix inequality~\eqref{LMI} and the estimate $\bY_\eta\leq \lambda_{\max}(\bY_\eta) \bI$.
Then, applying Gronwall's inequality yields
\begin{equation} \label{EE_Y_n}
 \|x_\eta(t)\|_{\bY_\eta}^2
\leq e^{-2\kappa_2 t/\lambda_{\max}(\bY_\eta)} \|x_\eta(0)\|_{\bY_\eta}^2 .
\end{equation}
Using that the self-adjoint operators $\bY_\eta$, $\eta\geq 1$ satisfy
$\lambda_{\min}(\bY_\eta) \bI \leq \bY_\eta\leq \lambda_{\max}(\bY_\eta) \bI$, we infer from~\eqref{EE_Y_n} that
\begin{equation} \label{EE}
 \|x_\eta(t)\|^2_\cH
\leq \frac{\lambda_{\max}(\bY_\eta)}{\lambda_{\min}(\bY_\eta)} e^{-2\kappa_2 t/\lambda_{\max}(\bY_\eta)} \|x_\eta(0)\|^2_\cH \, .
\end{equation}
Using~\eqref{EV_Yn} in~\eqref{EE} yields the claimed estimate~\eqref{unif-decay} with 
\begin{equation}\label{lambda0+alpha}
  \lambda_\eta:=\kappa_2/\lambda_{\max}(\bY_\eta) >0
  \quad \mbox{and} \quad \alpha:=\epsilon  \sqrt{\lambda_{\max}\big(\bJ_{2,1}\bJ_{2,1}^*\big)} \in(0,\frac12].
\end{equation}
For $\eta\geq 1$, the rates $\lambda_\eta$ are bounded from below uniformly w.r.t.~$\eta$ as
\begin{equation}\label{lambda0}
 \lambda_{\eta}
=\kappa_2/\lambda_{\max}(\bY_\eta)
\geq \lambda_0
 \quad \mbox{with} \quad \lambda_0 :=\kappa_2/\Big( 1+\epsilon \sqrt{\lambda_{\max}\big(\bJ_{2,1}\bJ_{2,1}^*\big)}\Big) .
\end{equation}
\end{proof}

%%%%%%%%%%%%%%%%%%%%%%%%%%%%%%%%%%%%%%%%%%%%%%%

\subsection{Proof of short-time behavior}\label{sec:short}
The following lemma is an $\eta$-uniform extension of Lemma 2.6 in \cite{AcArMeNi25}, see also Remark \ref{equiv-HCcond}.
This result was also used in Appendix C of \cite{AcArMeNi25}, but the proof was omitted there.
\begin{lemma}\label{2HC-cond}
Let $\bC=\bR-\bJ\in \cB(\cH)$ be an accretive operator that satisfies $\|\bR\|=1$ and $\bR+\bJ\bR\bJ^*\ge\kappa_1\bI$. Then 
$$
  \mR + \bC_\eta^* \mR \bC_\eta \ge \kappa_3 \mI\qquad \mbox{ for all } \eta \geq 1
$$
with $\kappa_3=\frac{3-\sqrt5}{2} \kappa_1$.
\end{lemma}

\begin{proof}
The self-adjoint operators $\bR$ and $\bR^3$ have a spectral decomposition w.r.t.\ the same spectral measure. From the assumption $\|\bR\|=1$, we hence obtain $0\le\bR^3\le\bR\le\bI$ and estimate for some $\epsilon=\epsilon(\eta)\in(0,1)$:
\begin{eqnarray*}
\mR + \bC_\eta^* \mR \bC_\eta 
&\ge& \epsilon (\bR+\bJ^*\bR\bJ)
+ (2-\epsilon)\bR^3 + (\eta^2-\epsilon) \bJ^*\bR\bJ -\eta(\bJ^*\bR^2+\bR^2\bJ) \\
&\ge& \epsilon\kappa_1\bI +\big(\sqrt{2-\epsilon}\bR-\sqrt{\eta^2-\epsilon}\bJ^*\big) \bR \big(\sqrt{2-\epsilon}\bR-\sqrt{\eta^2-\epsilon}\bJ\big) 
\ge \epsilon\kappa_1\bI\,,
\end{eqnarray*}
where $\epsilon$ must satisfy $\sqrt{2-\epsilon}\sqrt{\eta^2-\epsilon}=\eta$. Hence
$$
  \epsilon(\eta)=1+\frac{\eta^2}{2} - \sqrt{1+\eta^4/4}<1\,,
$$
which is monotonically increasing. Using $\epsilon(1)=\frac{3-\sqrt5}{2}$ gives the result.
\end{proof}
\smallskip

Since the following proof is just a small variant of \cite[Appendix C]{AcArMeNi25}, we give here only the key estimates and compare them to \cite{AcArMeNi25}. First, we note that $\eta\ge1$ corresponds to $|\bn|\ge1$ in \cite{AcArMeNi25}.

\begin{proof}[Proof of Theorem~\ref{thm1} (c)]
To derive the uniform estimate \eqref{Pn-decay}, we combine for each $\eta\ge1$ a short-term decay estimate for the initial phase $[0,\frac{\tau}{\eta}]$ (which shrinks w.r.t.~increasing $\eta$) with the long-term decay estimate \eqref{unif-decay} for the remaining time interval $[\frac{\tau}{\eta},\tau]$. Considering the form of \eqref{unif-decay}, we actually have the following three phases of estimates for $\|\bP_\eta(t)\|$, where the constants will be specified below (see also Figure \ref{fig:sketch_of_proof}):

\begin{figure}
  \centering
  \includegraphics[width=0.8\textwidth]%{Sketch_of_proof.pdf}
  {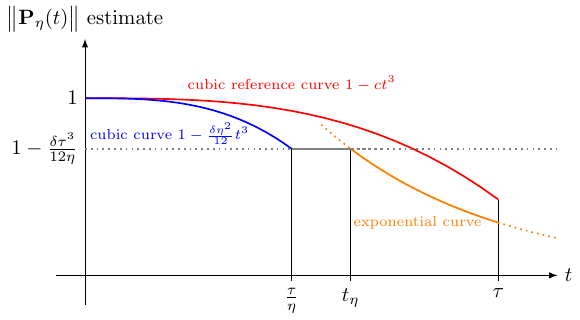}
  \caption{To derive the uniform estimate $\|\bP_\eta(t)\|_{\cB(\cH)} \le 1-c t^3$ for $0\le t\le \tau$ in~\eqref{Pn-decay}, we combine a short-term decay estimate for the initial phase $[0,\frac{\tau}{\eta}]$  (that shrinks w.r.t.\ $\eta$) with the long-term decay estimate \eqref{unif-decay} for the remaining time interval $[\frac{\tau}{\eta},\tau]$.}
  \label{fig:sketch_of_proof}
\end{figure}

\begin{enumerate}
\item algebraic estimate, obtained from Inequality (96) in \cite{AcArMeNi25}:
\begin{equation}\label{est1}
  \|\bP_\eta(t)\| \le 1 -\frac{\eta^2\delta}{12} t^3 \le 1-ct^3,\quad 0\le t\le\frac{\tau}{\eta},\quad \eta\ge1;
\end{equation}

\item constant estimate, obtained from Inequality (99) in \cite{AcArMeNi25}:
\begin{equation}\label{est2}
  \|\bP_\eta(t)\| \le 1 -\frac{\delta\tau^3}{12\eta} \le 1-ct^3, \quad \frac{\tau}{\eta} \le t\le t_\eta, \quad \eta\ge1;
\end{equation}

\item exponential estimate, obtained from Inequality\footnote{We remark that there is a typo in that inequality in \cite{AcArMeNi25}, and it is corrected here.} (100) in \cite{AcArMeNi25}:
\begin{equation}\label{est3}
  \|\bP_\eta(t)\| \le \left(1-\frac{\delta\tau^3}{12\eta}\right) \sqrt{1+\frac{2\alpha}{\eta-\alpha}} e^{-\lambda_0(t-\tfrac{\tau}{\eta})} \le 1-ct^3, \quad t_\eta\le t\le \tau, \quad \eta\ge r.
\end{equation}
\end{enumerate}

In these estimates, $\lambda_0$ and $\alpha$ were already defined in \eqref{lambda0}, \eqref{lambda0+alpha}, while $\alpha=\frac12$ in \cite{AcArMeNi25}. 
For $\eta\geq 1$, we consider the system~\eqref{IVP_eta} with $\bC_\eta =\bR -\eta\bJ$. 
By \eqref{est1b} and Lemma \ref{2HC-cond} there exist constants $\kappa_1, \kappa_3>0$ (independent of $\eta$) such that
\[ \mR + \eta^2 \bJ \mR \bJ^* \ge \kappa_1 \mI,\quad \mR + \bC_\eta^* \mR \bC_\eta \ge \kappa_3 \mI\qquad \mbox{ for all } \eta \geq 1. \]
This yields $\delta:=\min(\kappa_1/5,\kappa_3/2)$ by the same formula as in \cite{AcArMeNi25}. 

By assumption we have~$\norm{\bR}=1$ and $\|\bJ\| =: \beta$, while the analogous operators in \cite{AcArMeNi25} satisfy $\norm{\bR}=\|\bJ_{(1,0)}\|=1$. This will require minor modifications of the detailed estimates, using here $\theta:=1+\beta$. The monotonically increasing functions $\delta_1$, $\delta_3$ are now defined as
$$
   \delta_1(\tau_1) := \frac{e^{2\theta \tau_1} - 1 - 2\theta \tau_1}{\theta\tau_1}, \qquad 
   \delta_3(\tau_3):=\frac{e^{2\theta \tau_3} - 1 - 2\theta \tau_3 - 2\theta^2 \tau_3^2 - \frac{4}{3} \theta^3\tau_3^3}{\theta \tau_3^3}\,.
$$
They uniquely fix the constants $\tau_1$, $\tau_3>0$ by the equations
$$
  \delta_1(\tau_1)=\delta,\qquad \delta_3(\tau_3)=\frac{\delta}{12}\,,
$$
while $\tau_2$ is given by the same formula as in \cite{AcArMeNi25}:
$$
   \tau_2(\delta) := \frac{\sqrt{12 \delta}}{\inf\limits_{\substack{x \in \sphere_\cH\\ \| \sqrt \bR  x\| \le \sqrt{\delta}}} \| \sqrt \bR \bJ x\| + \sqrt{\delta}}\,,
$$
where $\sphere_\cH \coleq \{ x\in \cH : \norm{x}_\cH =1\}$.

The time intervals in the main estimates \eqref{est1}-\eqref{est3} are defined by 
\begin{equation}\label{def-tau1}
  \tau:=\min(\tau_1,\tau_2,\tau_3,1), \qquad t_\eta:=\frac\tau\eta +\frac{\ln\big(1+\frac{2\alpha}{\eta-\alpha}\big)}{2\lambda_0}\,,
\end{equation}
which coincides with \cite{AcArMeNi25}.

As in \cite{AcArMeNi25}, the constant $r>1$ for \eqref{est3} is uniquely defined via the equation
$$
  \sqrt{1+\frac{2\alpha}{r-\alpha}} e^{-\lambda_0\tfrac{r-1}{r}\tau}=1\,.
$$

Finally, the multiplicative constant in \eqref{Pn-decay} and \eqref{est1}-\eqref{est3} is given by
\begin{equation}\label{def-c1}
  c:= \frac{\delta}{12}\min\big[\big(1+\frac{1}{\lambda_0\tau}\big)^{-3},\frac{1}{r}]\,.
\end{equation}
Although not presented in this way, this last term also coincides with the result in \cite{AcArMeNi25}.
\end{proof}

\section{Conclusions}
For families of dissipative evolution equations, the uniform long- and short-time behavior of solutions has been studied, and decay estimates for the propagator norm are determined.

Future work will include the analysis when the hypocoercivity index is larger than $1$ and when the kernel of the Hermitian part is infinite dimensional.
\\

{\bfseries Acknowledgments. } 
The first two authors (FA, AA) were supported by the Austrian Science Fund (FWF) project \href{https://doi.org/10.55776/F65}{10.55776/F65}.

%--------------------------------------------------------

\end{document}